\documentclass[12pt]{amsart}
\usepackage[utf8]{inputenc}

\title[Weak sequential completeness]{On weak sequential completeness of spaces where weakly compact sets are super weakly compact}
\author[Z. Silber]{Zdeněk Silber}
\email{zdesil@seznam.cz}
\address{Institute of Mathematics of the Czech Academy of Sciences,
Žitná 25, 115 67 Prague 1, Czech Republic}


\usepackage{amsthm}
\usepackage{amsmath}
\usepackage{mathrsfs}
\usepackage{mathtools}
\usepackage[shortlabels]{enumitem}
\usepackage{verbatim}
\usepackage{amsfonts}
\usepackage{amssymb}
\usepackage[a-2u]{pdfx}
\usepackage{tikz}
\usepackage[most]{tcolorbox}
\usepackage{esint}
\usepackage{enumitem}
\usetikzlibrary{matrix,arrows}
\graphicspath{{images/}}

\newcommand{\N}{\mathbb{N}}

\newcommand{\R}{\mathbb{R}}

\newtheorem{theorem}{Theorem}[section]

\newtheorem{prop}[theorem]{Proposition}
\newtheorem{corollary}[theorem]{Corollary}

\newtheorem{question}{Question}

\theoremstyle{definition}
\newtheorem{definition}[theorem]{Definition}
\newtheorem{remark}[theorem]{Remark}

\newcommand{\norm}[1]{\left\lVert#1\right\rVert}

\begin{document}

\baselineskip=17pt

\begin{abstract}
    We show that every Banach space in which weakly compact sets are super weakly compact in automatically weakly sequentially complete answering a question from \cite{Silber2024}. In the proof we show how to build a weakly compact set which is not super weakly compact from an arbitrary nontrivial weakly Cauchy sequence using the notion of a summing subsequence of Rosenthal or Singer.
\end{abstract}

\maketitle

\section{Introduction}

Super weakly compact sets are a localization of superreflexivity in the same sense as weakly compact sets are a localization of reflexivity. Many papers were devoted to their study \cite{Cheng2018,Cheng2010,GrelierRaja2022,LancienRaja2022,Raja2016,Kun2021}, and many interesting properties of super weakly compact sets are known. The focus of this paper are spaces in which every weakly compact set is super weakly compact -- these include superreflexive spaces, Schur spaces, $L_1(\mu)$, $C(K)^*$ or preduals of JBW$^*$ triples \cite[Section 6]{LancienRaja2022}, Lipschitz-free spaces over superreflexive spaces and duals of rich subspaces of $C(K,\R^n)$ \cite[Remark 3.4., Theorem 3.6.]{Silber2024}, or Lipschitz-free spaces over some special kinds of Carnot-Carathéodory metric spaces \cite[Theorem 3.7.]{AliagaPerneckaQuero2024}.

In \cite{Silber2024}, the so-called property (R) was investigated -- a Banach space $X$ has property (R) if and only if every weakly precompact subset of $X$ is relatively super weakly compact. It can be easily seen that property (R) of $X$ is equivalent to the fact that $X$ is weakly sequentially complete and every relatively weakly compact subset of $X$ is relatively super weakly compact. As all the known Banach spaces where every weakly compact set is super weakly compact (see the list in the previous paragraph) are also weakly sequentially complete, the following question was asked \cite[Question 1]{Silber2024}:
\begin{question} \label{Question}
    Is there a Banach space in which every weakly compact set is super weakly compact but which is not weakly sequentially complete?
\end{question}

In this paper we give a negative answer to this question. Our method is based on the notion of a summing sequence -- a seminormalised basic sequence on which we can define the summing functional. This notion was used by Rosenthal under the name wide-$(s)$ sequence \cite{Rosenthal1994} (the same notion was originally used under the name $P^*$ sequence by Singer \cite{Singer1962}). By a result from \cite{Rosenthal1994}, every nontrivial weakly Cauchy sequence (that is weakly Cauchy but not weakly convergent sequence) has a summing subsequence, and thus we can think of nontrivial weakly Cauchy summing sequences as a kind of witness sequences for failure of weak sequential completeness. We describe a way how to use a nontrivial weakly Cauchy summing sequence in a Banach space $X$ to find a subset of $X$ which is weakly compact but not super weakly compact, which will answer Question \ref{Question}. It follows that the assumption of weak sequential completeness in the definition of property (R) is redundant.

\section{The result}

Let us first recall the main notions we will work with in the rest of the paper.

\begin{definition}
    Let $A$ be a subset of a Banach space $X$. We say that $A$ is \textit{relatively super weakly compact} if $A^{\mathcal{U}}$ is relatively weakly compact in $X^{\mathcal{U}}$ for every free ultrafilter $\mathcal{U}$ on $\N$. We say that $A$ is \textit{super weakly compact} if it is relatively super weakly compact and weakly closed.
\end{definition}

Recall that $X^{\mathcal{U}}$ is the quotient of $\ell_{\infty} (X)$ by the null space $N_{\mathcal{U}} = \{(x_n)_{n \in \N} \in \ell_\infty(X): \lim_\mathcal{U} x_n = 0\}$ and $A^{\mathcal{U}}$ is the subset of $X^{\mathcal{U}}$, whose elements have representatives $(x_n)_{n \in \N}$ where $x_n \in A$ for each $n \in \N$. For more details see e.g. \cite{LancienRaja2022}. The following proposition by Lancien and Raja \cite[Theorem 3.7.]{LancienRaja2022} is a characterization of convex super weakly compact sets.

\begin{prop} \label{Proposition-SWCC-char}
    Let $A$ be a bounded closed convex subset of a Banach space $X$. The following are equivalent:
    \begin{enumerate}[(i)]
        \item $A$ is super weakly compact;
        \item There is no $\theta > 0$ such that for every $n \in \N$ there are sequences $(x^n_k)_{k \leq n} \subseteq A$, $(f^n_k)_{k \leq n} \subseteq B_{X^*}$ such that $f^n_k(x^n_j) = 0$ if $k > j$ and $f^n_k(x^n_j) = \theta$ if $k \leq j$.
    \end{enumerate}
\end{prop}

\begin{definition}
    A sequence $(x_n)_{n \in \N}$ in a Banach space $X$ is \textit{weakly Cauchy} if for every $x^* \in X^*$ the sequence of scalars $(x^*(x_n))_{n \in \N}$ is convergent. We say that $(x_n)_{n \in \N}$ is \textit{nontrivial weakly Cauchy} if it is weakly Cauchy but not weakly convergent.    
    A Banach space $X$ is \textit{weakly sequentially complete} if every weakly Cauchy sequence in $X$ is weakly convergent.
\end{definition}

Application of the uniform boundedness principle yields that a sequence $(x_n)_{n \in \N}$ in $X$ is weakly Cauchy if and only if there is $x^{**} \in X^{**}$ such that $(x_n)_{n \in \N}$ converges to $x^{**}$ in the weak$^*$ topology of $X^{**}$, and it is nontrivial weakly Cauchy if this $x^{**}$ is an element of $X^{**} \setminus X$.

\begin{definition}
    A seminormalised basic sequence $(x_n)_{n \in \N}$ in a Banach space $X$ is called \textit{summing} if there is a functional $x^* \in X^*$ such that $x^*(x_n) = 1$ for every $n \in \N$. The functional $x^*$ is called the \textit{summing functional} of $(x_n)_{n \in \N}$.
\end{definition}

It can be easily seen that a seminormalised basic sequence $(x_n)_{n \in \N}$ is summing if and only if it dominates the summing basis of $c_0$ (that is the mapping $\sum_{n \in \N} a_n x_n \in \overline{\operatorname{span}} (x_n)_{n \in \N} \mapsto \sum_{n \in \N} a_n f_n \in c_0$ is bounded, where $(f_n)_{n \in \N}$ is defined by $f_n = e_1 + \cdots + e_n$ and $(e_n)_{n \in \N}$ is the canonical basis of $c_0$). Our notion of a summing sequence differs from the notion of an $(s)$ sequence of Rosenthal \cite[Definition 2.1.]{Rosenthal1994} in the fact that we do not apriori assume the sequence to be nontrivial weakly Cauchy. In the terminology of \cite{Rosenthal1994}, our summing sequences are precisely the wide-$(s)$ sequences, and in the terminology of \cite{Singer1962}, our summing sequences are precisely the $P^*$ sequences.

The following proposition is taken from \cite[Proposition 2.2.]{Rosenthal1994}.

\begin{prop} \label{Proposition-Sum-Subseq}
    Let $(x_n)_{n \in \N}$ be a nontrivial weakly Cauchy sequence in a Banach space $X$. Then $(x_n)_{n \in \N}$ has a summing subsequence.
\end{prop}

\begin{remark}
    A Banach space $X$ is not reflexive if and only if it contains a summing sequence. This has been proven by Singer \cite[Theorem 2.]{Singer1962} for $X$ with a basis and by Pe{\l}czy{\'n}ski \cite{Pelczynski1962} for general $X$. We will now show an alternative proof.
    
    First note that a summing sequence cannot have a weakly convergent subsequence (as the only possible weak cluster point of a seminormalised basic sequence is zero, see e.g. \cite[Lemma 1.6.1.]{nigel2006topics}, and the summing functional witnesses that no subsequence of the summing sequence can be weakly null), and thus a Banach space containing a summing sequence cannot be reflexive. On the other hand, if a Banach space $X$ is not reflexive, we can find a seminormalised sequence without weakly convergent subsequences. By Rosenthal's $\ell_1$ theorem \cite{Rosenthal1974} we can pass to a subsequence which is either weakly Cauchy or equivalent to the canonical basis of $\ell_1$. In the first case, we get a nontrivial weakly Cauchy sequence which has a summing subsequence by Proposition \ref{Proposition-Sum-Subseq}. In the second case, we just need to realize that the canonical basis of $\ell_1$ is clearly summing.
\end{remark}

\begin{prop} \label{Proposition-Sum-Fctionals}
    Let $(x_n)_{n \in \N}$ be a summing sequence in a Banach space $X$. Then there is a uniformly bounded sequence of functionals $(s_n)_{n \in \N} \subseteq X^*$ such that $s_n(x_m) = 1$ if $n \leq m$ and $s_n(x_m) = 0$ if $n > m$.
\end{prop}

\begin{proof}
    By the Hahn-Banach theorem we can assume that $(x_n)_{n \in \N}$ is a basis of $X$, that is, we can assume that $X = \overline{\operatorname{span}}(x_n)_{n \in \N}$. For $n \in \N$ define $x^*_n \in X^*$ by the formula
    \begin{align*}
        x^*_n(x) = \sum_{j=1}^n a_j, \hspace{1cm} x = \sum_{j=1}^\infty a_j x_j \in X.
    \end{align*}
    Then it is easy to check that $(x_n^*)_{n \in \N}$ weak$^*$ converges to the summing functional $x^*$ of $(x_n)_{n \in \N}$ in $X^*$. In particular, $(x_n^*)_{n \in \N}$ is bounded. Set $s_1 = x^*$ and for $n > 1$ set $s_n = x^* - x_{n-1}^*$. Then $(s_n)_{n \in \N}$ clearly satisfies the conclusion of the proposition.
\end{proof}

\begin{theorem}
    Let $X$ be a Banach space where relatively weakly compact sets are relatively super weakly compact. Then $X$ is weakly sequentially complete.
\end{theorem}

\begin{proof}
    Suppose that $X$ is not weakly sequentially complete. Then there is a nontrivial weakly Cauchy sequence $(x_n)_{n \in \N}$ in $X$. Using Proposition \ref{Proposition-Sum-Subseq} we can without loss of generality assume that $(x_n)_{n \in \N}$ is summing and denote by $(s_n)_{n \in \N}$ the functionals from Proposition \ref{Proposition-Sum-Fctionals}. Let us define
    \begin{align*}
        y^n_j = x_{n+j} - x_n, \hspace{1cm} n \in \N, \; j \leq n.    
    \end{align*}
    Then $\{y^n_j: n \in \N, j \leq n\}$ is a relatively weakly compact subset of $X$. Indeed, it follows from the fact that $(x_n)_{n \in \N}$ is weakly Cauchy that $\{y^n_j: n \in \N, j \leq n\}$ is actually a weakly null sequence. Hence, $A = \overline{\operatorname{conv}} \{y^n_j: n \in \N, j \leq n\}$ is weakly compact by the Krein-Šmulian theorem. We will use Proposition \ref{Proposition-SWCC-char} to show that $A$ is not super weakly compact. Define $(g^n_j)_{j \leq n}$, $n \in \N$, by $g^n_j = s_{n+j}$, then
    \begin{align*}
        g^n_j(y^n_i) = s_{n+j} (x_{n+i} - x_n) = s_{n+j}(x_{n+i}) =
        \begin{cases}
		  1, & \text{if $j \leq i$}\\
            0, & \text{if $j > i$}.
		 \end{cases}
    \end{align*}
    Hence, if we set $c = \sup \norm{s_n} = \sup \norm{g^n_j}$, and $f^n_j = c^{-1} g^n_j$ for $j \leq n$, then $(y^n_j)_{j \leq n} \subseteq A$ and $(f^n_j)_{j \leq n} \subseteq B_{X^*}$, $n \in \N$, witness the failure of Proposition \ref{Proposition-SWCC-char} (ii) with $\theta = c^{-1}$. It follows by Proposition \ref{Proposition-SWCC-char} that $A$ is not relatively super weakly compact. Hence, it is not true that relatively weakly compact sets in $X$ are relatively super weakly compact, and the theorem follows.
\end{proof}

Recall the following definition of property (R) from \cite{Silber2024}.

\begin{definition}
    We say that a Banach space $X$ has property (R) if every weakly precompact subset of $X$ is relatively super weakly compact.
\end{definition}

It can be easily seen that a Banach space $X$ has property (R) if and only if it is weakly sequentially complete and every weakly compact subset of $X$ is super weakly compact. In the light of the preceding theorem, the assumption of weak sequential completeness in the definition of property (R) is redundant -- we have:

\begin{corollary}
    Let $X$ be a Banach space. The following are equivalent:
    \begin{enumerate}
        \item $X$ has property (R);
        \item Every weakly compact subset of $X$ is super weakly compact.
    \end{enumerate}
\end{corollary}

\bibliographystyle{siam}
\bibliography{bibliography}

\end{document}